\documentclass[12pt]{article}
\usepackage{amsmath}
\usepackage{amssymb}
\usepackage{nccmath}
\usepackage{amsthm}
\usepackage{graphicx}
\usepackage{enumerate}
\usepackage{csquotes}
\usepackage{bbold}
\usepackage{amsmath}
\usepackage{amssymb}
\usepackage{nccmath}
\usepackage{amsthm}
\usepackage{url} 

\usepackage{csquotes}
\MakeOuterQuote{"}

\usepackage{natbib}
\allowdisplaybreaks

\newcommand{\blind}{1}

\theoremstyle{plain} 
\newtheorem{theorem}{Theorem}
\newtheorem{lemma}{Lemma}
\newtheorem{definition}{Definition}[section]

\begin{document}

\def\spacingset#1{\renewcommand{\baselinestretch}%
{#1}\small\normalsize} \spacingset{1}


\if1\blind
{
  \title{\bf The Lindeberg\textendash Feller and Lyapunov Conditions in Infinite Dimensions}
  \author{Julian Morimoto \thanks{\textit{This project has not been funded.}}\hspace{.2cm}}
  \maketitle
} \fi

\if0\blind
{
  \bigskip
  \bigskip
  \bigskip
  \begin{center}
    {\LARGE\bf The Effect of Sample Size and Missingness on Inference with Missing Data
}
\end{center}
  \medskip
} \fi

\bigskip
\begin{abstract}
This paper makes 3 contributions. First, it generalizes the Lindeberg\textendash Feller and Lyapunov Central Limit Theorems to Hilbert Spaces by way of $L^2$. Second, it generalizes these results to spaces in which sample failure and missingness can occur. Finally, it shows that satisfaction of the Lindeberg\textendash Feller Condition in such spaces guarantees the consistency of all inferences from the partial functional data with respect to the completely observed data. These latter two results are especially important given the increasing attention to statistical inference with partially observed functional data. This paper goes beyond previous research by providing simple boundedness conditions which guarantee that \textit{all} inferences, as opposed to some proper subset of them, will be consistently estimated. This is shown primarily by aggregating conditional expectations with respect to the space of missingness patterns. This paper appears to be the first to apply this technique.
\end{abstract}

\noindent%
\vfill

\newpage
\spacingset{1.9} 

\section{Introduction}
The Lindeberg\textendash Feller and Lyapunov Central Limit Theorems are two important results in classical parametric (i.e. finite-dimensional) inference. They give simple conditions that, if satisfied, guarantee the asymptotic normality of a sum of independent, but not necessarily identically distributed random variables \cite{durrett_probability_2019}. These results help researchers better understand the distribution of error terms, thereby enabling them to make clear statements about how uncertain they are about some estimate. 

Given the usefulness of the Lindeberg\textendash Feller and Lyapunov theorems in finite dimensional settings, a natural question to ask is whether these results also hold in infinite dimensions. This is particularly important for the typical problem in non-parametric inference in which the "parameter" of interest is not a vector in $\mathbb{R}^n$, but a function \cite{gine_mathematical_2021}. Solving this problem is becoming increasingly important in modern-day data analysis, where new technologies produce large datasets with many more dimensions than samples \cite{sur_modern_2019}. Common examples of such data are those that can be modeled as curves or functions, such as ECG data of patients over time, spectrometry recordings over various wavelengths, and temperature recordings over different places \cite{chakraborty_spatial_2014}. Researchers have thus spent considerable time studying asymptotic normality in the infinite-dimensional context to help people analyze such datasets \cite{gine_mathematical_2021, kandelaki_central_1964, osipov_infinite-dimensional_1985}. This paper builds on these well-studied inquiries by offering infinite-dimensional analogues of two well-known asymptotic normality results in finite-dimensions: the Lindeberg\textendash Feller and Lyapunov Central Limit Theorems.

To do this, it starts by reformulating the Lindeberg\textendash Feller and Lyapunov conditions using the terminology of Hilbert Spaces, specifically $L^2$. This is the first main result of this paper. This reformulation adheres closely to classical procedures for non-parametric inference that rely on Hilbert Spaces \cite{gine_mathematical_2021, kraus_components_2015, kraus_inferential_2019}, as well as to previous attempts to generalize these theorems to infinite-dimensional contexts \cite{kandelaki_central_1964}. Then, following the approach of \cite{durrett_probability_2019}, this paper analyzes the asymptotic properties of the characteristic functionals of sums of infinite-dimensional random variables that are independent but not necessarily identically distributed. However, this paper departs from \cite{durrett_probability_2019} in that it uses an infinite-dimensional alternative to the L\'evy Continuity Theorem given in \cite{feldman_short_1965, gross_integration_1960}. This is necessary as the L\'evy Continuity Theorem does not hold in infinite dimensions where the tightness of measures is not guaranteed. Indeed, this is a key reason why the extension of these two limit theorems to infinite-dimensions is not a trivial exercise. 

Furthermore, this paper goes beyond existing scholarship by providing two additional results. First, these two theorems also hold in spaces in which sample failure/missingness can occur. Second, if the Lindeberg\textendash Feller condition is satisfied in these spaces, then all inferences drawn from the partially observed functional data will be asymptotically identical to those that would have been drawn from the complete data. These are the second and third main results of this paper, respectively. They are especially important in light of the recently heightened attention to statistical inference with partially observed functional data \cite{kraus_components_2015, delaigle_approximating_2016, yang_asymptotic_2020, kneip_optimal_2020, kraus_inferential_2019, wang_does_2019}.

These latter two results also go well beyond this scholarship in a crucial way. Previous scholarship provided loosely connected conditions that guarantee the consistency of a proper subset of inferences such as the mean function and the covariance operators. In contrast, this paper provides simple boundedness conditions that guarantee the consistency of \textit{all} possible inferences that one can draw from a partially observed dataset. All inferences will be asymptotically identical to those one would have drawn from a complete one. Under these conditions, the distribution of the sum of partially observed high-dimensional random elements will be asymptotically equivalent to the distribution that would have been observed without missingness/sample failure. This is shown by aggregating conditional expectations with respect to the space of missingness patterns, a technique that appears to have not been used in the literature on inference with partial functional data. Thus, \textit{all} possible inferences that one can draw from the distribution generated by the partially observed elements will approximate the "true" inferences. These include, not only mean functions, covariance operators and eigenelements, but also hypothesis tests results and confidence intervals \cite{kraus_components_2015, kraus_inferential_2019}. 

While previous scholarship suggests that these results may be generalized to Hilbert Spaces other than $L^2$ (such as Reproducing Kernel Hilbert Spaces and Sobolev spaces) \cite{kraus_components_2015, kraus_inferential_2019}, this paper follows the approach of \cite{kraus_components_2015, kraus_inferential_2019} by focusing on $L^2$ spaces. This is for simplicity of presentation, and to introduce some techniques, such as aggregating conditional expectations with respect to the space of missingness patterns, that can be used in future research to prove analogous results in other kinds of Hilbert Spaces. 

\subsection{The Results of Kandelaki and Sozanov, Gin\'e and Le\'on, and Others}

\subsubsection{Kandelaki and Sozanov}

The first main result in this paper is similar to the main result in \cite{kandelaki_central_1964}; however, there is a crucial difference. \cite{kandelaki_central_1964} proved a version of the Lindeberg\textendash Feller Theorem that requires an extra condition that the result in this paper does not require. The Theorem in \cite{kandelaki_central_1964} requires that the Lindeberg\textendash Feller Condition be satisfied by way of a normalizing sequence of bounded linear operators, $\{A_n\}_{n \in \mathbb{N}}$. Formally, the version of the Lindeberg\textendash Feller condition given in \cite{kandelaki_central_1964} is: 

$$\forall \varepsilon > 0, \lim_{n \to \infty}\sum_{k=1}^n\int_{\|A_n h\|_{L^2(\mu)} \geq \varepsilon}\|A_nh\|_{L^2(\mu)}^2 dQ_{X_k}(h) = 0$$

where

\begin{enumerate}
    \item $h$ is an element of a Hilbert Space, $H$
    \item $\{X_k\}_{k \in \mathbb{N}}$ is a sequence of independent $H$-valued random elements
    \item $\{A_n\}_{n \in \mathbb{N}}$ is a sequence that normalizes $X_k$ with respect to a certain linear, symmetric, non-negative, completely continuous operator with finite trace
    \item $Q_{X_k}$ is the probability measure defined by $X_k$
\end{enumerate}

This is different from the classical version of the Lindeberg\textendash Feller condition which does not require a specific transformation of the random variables: 

$$\forall \varepsilon > 0, \lim_{n \to \infty} \sum_{m=1}^n \int_{\|X_{n,m}\|_2 \geq \varepsilon}\|X_{n,m}\|_2^2dP_{n,m} = 0$$ 

where $\{X_{n,m}\}_{n,m}$ is a sequence of independent random elements, and $P_{n,m}$ is the probability measure induced by $X_{n,m}$. In contrast to \cite{kandelaki_central_1964}, this paper proves an infinite-dimensional version of the Lindeberg\textendash Feller Theorem that eschews this additional requirement regarding certain normalizing linear operators. 

Proving a Hilbert Space version of the Lindeberg\textendash Feller Theorem that does not require a sequence of normalizing linear operators has two important consequences. First, it results in a theorem that is more general than that provided in \cite{kandelaki_central_1964}. This paper shows that \textit{any} distribution of a sum of independent random variables that satisfies the Lindeberg\textendash Feller condition will be asymptotically normal\textemdash even those that did not first undergo the transformations required by the Theorem in \cite{kandelaki_central_1964}. Second, the Theorem in this paper offers a much simpler set of conditions that need to be checked to obtain asymptotic normality. Researchers do not have to find or construct a sequence of bounded linear operators that normalize the random variables with respect to a certain linear, symmetric, non-negative, completely continuous operator with finite trace, which \cite{kandelaki_central_1964} requires. As long as a set of random variables satisfies the Lindeberg\textendash Feller condition, their sum will be asymptotically normal.

\subsubsection{Gin\'e and Le\'on}

The first main result in this paper is also similar to Corollary 4.3 in \cite{gine_central_1980}, however, the latter result also requires an additional, and quite strong boundedness condition that the main result in this paper does not require. A simplified version of the boundedness condition in Corollary 4.3 (condition iii; and condition ii in Theorem 4.2) is as follows: 

\begin{align*}
\lim_{n \to \infty} \lim_{N \to \infty} \sum_{m=1}^{n} \sum_{i = N}^{\infty} \mathbb{E}_{n,m}[\langle \chi_{n,m}, e_i \rangle ^2] = 0
\end{align*}

Where $\{e_i\}_{i \in \mathbb{N}}$ is some orthonormal basis of a Hilbert Space, $H$, and where \\ $\{\chi_{n,m}\}_{n \in \mathbb{N}, m \leq n}$ is a collection of independent random elements of $H$ with expectation $\textbf{0}$. What this condition essentially means is that the sum of the sizes of the "directional tails" of the Hilbert Space random elements must tend to zero in a particular way. 

While this condition is useful in that it guarantees a kind of sparsity, it may be difficult to check in practice for two reasons. First, one must find the right orthonormal basis, and order it in the right way, to verify this bound. Finding this basis may not be an easy computational task. Second, one must essentially verify that the random elements become very sparse, very quickly. This is because the double limit implies that the summation must tend to zero even when $n$ grows very quickly relative to $N$. The rate at which $n$ grows can be thought of as the rate at which the "directional tails" of the Hilbert Space random elements are being aggregated, whereas the rate at which $N$ grows can be thought as the rate of shrinkage of the number of these directional tails. Thus, even in cases where the number of directions of the random Hilbert Space elements appears to barely shrink relative to the number of the "directional tail sizes" being added, the total summation of the size of these tails must converge to zero. This condition can be quite strong, and it may be computationally difficult to check for all relative rates of $N$ and $n$. In contrast, the Theorem in this paper does not require such a bound, and thus avoids these practical difficulties associated with Corollary 4.3 in \cite{gine_central_1980}. 

Similar results are explored in other works, though they are not as general or as easily applicable as the main results of this paper. Chapter 2 in \cite{prokhorov_limit_2000} explores various central limit theorems in Hilbert Spaces, but is limited to independent and identically distributed random elements, whereas this paper does not require the elements to be identically distributed. Theorem 5.9 of \cite{gine_mathematical_2021} is a central limit theorem for Banach Spaces, however the authors note that the conditions for this theorem are hard to verify in general (indeed, the mechanics of this theorem are similar to those in Corollary 4.3 discussed above). In contrast, this paper provides a simple set of easily verifiable boundedness conditions that guarantee asymptotic normality. 

\section{Notation and Setup}
\label{notset}

The setup is as follows. All integrals are Bochner integrals unless stated otherwise. All references to the Fubini\textendash Tonelli Theorem are with respect to its formulation for Bochner integrals \cite{mikusinski_fubini_1978}. All Banach spaces on which the Bochner Integral operates are assumed to satisfy the Radon\textendash Nikodym Property. 

\begin{enumerate}
    \item $\{\chi_{n,m}\}_{n \in \mathbb{N}, m \leq n}$ is a collection of independent random elements of the separable Hilbert Space, $H = L^2([0,1], \mathcal{B}([0,1]), \mu)$, defined on the probability spaces, $\{( \Omega, \mathcal{F}, P_{n,m})  \}_{n \in \mathbb{N}, m \leq n}$, such that $\chi_{n,m} : \Omega \times [0,1] \to \mathbb{R}$, $\mathbb{E}_{n,m}[\chi_{n,m}(w, \cdot)] = \\ \int_{w \in \Omega}\chi_{n,m}(w,\cdot)dP_{n,m} = \mathbf{0}$, and $\Omega$ is separable. $\mu$ is Lebesgue measure and independent of $P_{n,m}$. 

    \item $\forall n \in \mathbb{N}, m \leq n, \int_{\Omega \times [0,1]^2}|\chi_{n,m}(w, \cdot) \bigotimes \chi_{n,m}(w, \cdot)|d(P_{n,m} \times \mu \times \mu) < \infty$
    
    \item Consistent with other literature on non-parametric inference, $\forall n \in \mathbb{N}$, $m \leq n, (x,y) \in [0,1]^2$,  $\mathcal{K}_{n,m}(x, y) := \mathbb{E}_{n,m}[\chi_{n,m}(w, x, y)] := \mathbb{E}_{n,m}[\chi_{n,m}(w, x)\chi_{n,m}(w, y)]$ is the covariance of the Hilbert Space random variable $\chi_{n,m}$ at $(x,y)$ \cite{gine_mathematical_2021}. This is equivalent to the definition of covariance given in \cite{hsing_theoretical_2015}: $\mathcal{K}^*_{n,m} := \mathbb{E}_{n,m}[(\chi - m_{n,m}) \bigotimes (\chi - m_{n,m})] = \int_{\Omega} (\chi - m_{n,m}) \bigotimes (\chi - m_{n,m}) dP_{n,m}$, where $m_{n,m} = \mathbb{E}_{n,m}[\chi] := \int_{\Omega} \chi dP_{n,m}$ and $\bigotimes$ is the tensor product operator. 
\end{enumerate}

When a realization of a Hilbert Space random element, $\xi \in H$, is integrated over the space $[0,1]^2$, it means that $\xi \bigotimes \xi$ is integrated over $[0,1]^2$. Thus, for $\xi_1, \xi_2 \in H$, $\langle \xi_1, \xi_2 \rangle_{L^2(\mu \times \mu)} = 
\int_{(x,y) \in [0,1]^2} \xi_1(x) \xi_1(y) \xi_2(x) \xi_2(y)$, and $\|\cdot\|_{L^2(\mu \times \mu)}$ adopts an analogous meaning.

All proofs are in the Appendix unless stated otherwise. 

\section{Asymptotic Normality}
This section discusses the first main result of this paper: Hilbert Space analogues of the Lindeberg\textendash Feller and Lyapunov Central Limit Theorems. The classical Lindeberg\textendash Feller and Lyapunov Theorems require the L\'evy Continuity Theorem. But this Theorem fails in infinite dimensional settings, where tightness of measures is not guaranteed. Lemmas \ref{lemma:grossconthm} and \ref{lemma:taucont} are meant to address this unique infinite-dimensional challenge. Lemma \ref{lemma:grossconthm} provides an alternative to the classical L\'evy Continuity Theorem in Hilbert Spaces, and Lemma \ref{lemma:taucont} shows that the natural extension of characteristic functions in the classical Lindeberg-Feller Theorem to Hilbert Spaces has a nice topological structure that allows for the application of Lemma \ref{lemma:grossconthm}. Lemmas \ref{lemma:grossconthm} and \ref{lemma:taucont} require the following definitions from \cite{feldman_short_1965, gross_integration_1960}.

\begin{definition}[Characteristic Functional/Characteristic Function]
For an $H$-valued random element $\chi$, the characteristic functional of $\chi$ is given by $\phi_{\chi}(\psi):= \mathbb{E}[e^{i\langle \psi, \chi \rangle_{L^2(\mu)}}]$ where $\psi \in H$. This will also be referred to as the characteristic function of $P$, where $P$ is the probability measure that characterizes the distribution of $\chi$.
\end{definition}

\begin{definition}[The $H_2$ Topology]
The $H_2$ topology of $H$ is the weakest topology on $H$ such that all Hilbert-Schmidt operators on $H$ are continuous from $H_2$ to $H_0$.
\end{definition}

\begin{definition}[Uniformly $\tau$-continuous near zero]
A map $\phi : H \to \mathbb{R}$ is uniformly $\tau$-continuous near zero if there exists a sequence $\{A_n\}_{n \in \mathbb{N}}$ of Hilbert-Schmidt operators such that $\|A_n\|_2 \to 0$ (where $\|\cdot\|_2$ is the Hilbert-Schmidt norm), and such that $\phi$ is uniformly continuous on the topology $H_2$ on each of the sets $U_n = \{\psi \in H : \|A_n(\psi)\| < 1\}$.
\end{definition}

The continuity results that follow from these definitions will now be stated.

\begin{lemma}[Gross' Continuity Theorem]
\label{lemma:grossconthm}
Let $\{\phi_j\}_{j \in \mathbb{N}}$ be characteristic functions of some probability measures $\{P_j\}_{j \in \mathbb{N}}$ on a Hilbert Space. If $\exists \phi : H \to \mathbb{R}$, a function that is uniformly $\tau-$continuous near 0 such that $\phi_j \to \phi$ in probability, then $P_j \to P$ weakly where $\phi$ is the characteristic function of $P$. 
\end{lemma}

\begin{proof}
Proof given in Theorem 2 of \cite{feldman_short_1965}.
\end{proof}

\begin{lemma}
\label{lemma:taucont}
If $\exists \varsigma \in L^2 (\mu \times \mu) : \sum_{m=1}^n \mathcal{K}_{n,m}$ converges to $\varsigma$ with respect to $\| \cdot \|_{L^2(\mu \times \mu)}$, then, $G: L^2(\mu \times \mu) \to \mathbb{R}$ such that $G(\psi) = \langle \varsigma, \psi \rangle _{L^2(\mu \times \mu)}$ is uniformly $\tau$-continuous near zero.
\end{lemma}

The Lindeberg\textendash Feller condition and Theorem  will now be stated. 

\begin{definition}[The Lindeberg\textendash Feller Condition on Hilbert Spaces]
$\forall \varepsilon > 0, \\ \lim_{n \to \infty} \sum_{m=1}^n \mathbb{E}_{n,m}[ \|\chi_{n,m} \|_{L^2(\mu)}^2 :  \|\chi_{n,m} \|_{L^2(\mu)} > \varepsilon ] = 0$
\end{definition}

\begin{theorem}[The Lindeberg\textendash Feller Theorem on Hilbert Spaces]
\label{thm:lfnorm}
Suppose

\begin{enumerate}
    \item $\exists \varsigma \in L^2 (\mu \times \mu) : \sum_{m=1}^n \mathcal{K}_{n,m}$ converges to $\varsigma$ with respect to $\| \cdot \|_{L^2(\mu \times \mu)}$, and $\lim_{n \to \infty}\sum_{m=1}^n \mathbb{E}_{n,m}[\|\chi_{n,m}\|_{L^2 (\mu)}^2] > 0$. 
    \item $\forall \varepsilon > 0,  \lim_{n \to \infty} \sum_{m=1}^n \mathbb{E}_{n,m}[ \|\chi_{n,m} \|_{L^2(\mu)}^2 :   \|\chi_{n,m} \|_{L^2(\mu)} > \varepsilon ] = 0 $ \\ (The Lindeberg\textendash Feller Condition).
\end{enumerate}

Then, $S_n = \sum_{m=1}^n\chi_{n,m}$ is asymptotically normal with expectation $\mathbf{0}$ and covariance $\varsigma$.
\end{theorem}

The Hilbert Space version of the Lyapunov condition will now be stated, and it will be shown that this condition implies the Lindeberg\textendash Feller condition for Hilbert spaces. Thus, the former (which can be much easier to check) can replace the latter in Theorem \ref{thm:lfnorm}.

The Lyapunov Condition is as follows: 

\begin{definition}[The Lyapunov Condition on Hilbert Spaces]
There exists $\delta > 0$ such that $\lim_{n \to \infty} \sum_{m=1}^{n} \mathbb{E}_{n,m}[\|\chi_{n,m}\|_{L^2(\mu)}^{2 + \delta}] = 0$.
\end{definition}

The next Theorem shows that the Lyapunov condition implies the Lindeberg\textendash Feller condition. 

\begin{theorem}[The Lyapunov Theorem on Hilbert Spaces]
\label{thm:lyap}
If there exists $\delta > 0$ such that $\lim_{n \to \infty} \sum_{m=1}^{n} \mathbb{E}_{n,m}[\|\chi_{n,m}\|_{L^2(\mu)}^{2 + \delta}] = 0$, then $\forall \varepsilon > 0$, 

\begin{align*}
    \lim_{n \to \infty} \sum_{m=1}^n \mathbb{E}_{n,m}[ \|\chi_{n,m} \|_{L^2(\mu)}^2 :  \|\chi_{n,m} \|_{L^2(\mu)} > \varepsilon ] = 0
\end{align*}

\end{theorem}

\section{Missingness and Sample Failure}

This section shows the second and third main results of this paper. First, Theorems \ref{thm:lfnorm} and \ref{thm:lyap} are shown to hold on spaces where functional data are partially observed. Second, it is shown that the satisfaction of the Lindeberg\textendash Feller condition in the space of partial observations guarantees the consistency of all inferences drawn from partial functional data. The notation and setup for this regime are as follows:

\subsection{Constructing the Space of Partially Observed Functions}
The elements of the space of partially observed functions will essentially be identical to the elements of the space of completely observed functions, except that some parts of the functions will be labeled as "missing" and replaced with imputed values. Let $ \{\chi_{n,m}\}_{n \in \mathbb{N}, m \leq n}$ be as in Section \ref{notset}. Let $ \{\chi'_{n,m}\}_{n \in \mathbb{N}, m \leq n}$ be a collection of independent random elements of $H$, so that each $\chi_{n,m}'$ is a function from $\Omega \times \mathcal{M} \times [0,1]$ to $ \mathbb{R}$, and is equipped with joint probability measure $P_{n,m} \times \nu_{n,m} \times \mu$, such that:
    
 \begin{enumerate}
        \item $\mathcal{M}$ is a measure space over $\{0,1\}^{[0,1]}$, equipped with probability measures \\ $\{\nu_{n,m}\}_{n \in \mathbb{N}, m \leq n}$, that are independent of Lebesgue measure $\mu$, and such that $\nu_{n,m}$ and $P_{n,m}$ need not be independent. $\mathcal{M}$ is essentially the space of missingness patterns and consists of functions over $[0,1]$ that are equal to $0$ at $x$ when the realized Hilbert Space random element (i.e., when there is some fixed $w \in \Omega$) is "missing" at $x$, and $1$ otherwise. Note that $\mathcal{M}$ is the product of continuum many separable spaces and is therefore separable by the Hewitt\textendash Marczewski\textendash Pondiczery Theorem, thus making $\Omega \times \mathcal{M}$ a separable Hilbert Space.
    
        %
        
        \item For all $n \in \mathbb{N}, m \leq n, w \in \Omega$ and $M \in \mathcal{M}$, $O(\chi_{n,m}'(w,M, \cdot)) := \chi_{n,m}'(w, M, \cdot )\big|_{X_1(M)}$ where $X_1(M) := \{x \in [0,1] : M(x) = 1\}$. This will be the "observed part" of $\chi_{n,m}'(w,M, \cdot)$. Let $X_0(M) := \{x \in [0,1] : M(x) = 0\}$. $\chi_{n,m}'(w, M, \cdot)\big|_{X_0(M)}$ will be the "imputed part" of $\chi_{n,m}'(w,M, \cdot)$.
        %
        
        \item For all $n \in \mathbb{N}, m \leq n, w \in \Omega$ and $M \in \mathcal{M}$,
        
        \begin{align*}
            \chi_{n,m}'(w, M, \cdot) := \chi_{n,m}(w, \cdot) \mathbb{1}_{X_1(M)} + (\mathbb{E}_{n,m}[\chi_{n,m}(w', \cdot) | O(\chi_{n,m}'(w', M, \cdot)) =
            \\
            O(\chi_{n,m}'(w, M, \cdot))] + \mathcal{G}(0, \mathcal{K}_{n,m}^{w,M})(\cdot)) \mathbb{1}_{X_0(M)} 
        \end{align*}
        
        where $\mathcal{G}(0, \mathcal{K}_{n,m}^{w,M})(\cdot)$ denotes a realization of a Gaussian process centered at $\textbf{0}$ with conditional covariance $\mathcal{K}_{n,m}^{w,M}(\cdot, \cdot) := \mathbb{E}_{n,m}[\chi_{n,m}(w,\cdot, \cdot) |O(\chi_{n,m}'(w', M, \cdot)) \\  = O(\chi_{n,m}'(w, M, \cdot)) ]$.  This process essentially imputes the missing portions of the functional data based on the observed portions. The mechanics of this imputation are discussed further below. 
        
        \item For all $n \in \mathbb{N}, m \leq n,$ the covariance operator of $\chi_{n,m}'$ is given by $\mathcal{K}_{n,m}' : = \mathbb{E}_{n,m}'[\chi_{n,m}'(w, M, \cdot, \cdot)] := \int_{\Omega \times \mathcal{M}} \chi_{n,m}'(w, M, \cdot, \cdot) $.

        
\end{enumerate}

\subsection{Assumptions about the Missingness Mechanism}
    
For all $n \in \mathbb{N},$ and $ m \leq n$, $\chi_{n,m}'$ is said to be Missing at Random (MAR) when $\forall w_1, w_2 \in \Omega $ and $ M \in \mathcal{M}$ such that $O(\chi_{n,m}'(w_1, M, \cdot)) =  O(\chi_{n,m}'(w_2, M, \cdot))$, it is the case that $\nu_{n,m} (M | w_1) = \nu_{n,m} (M | w_2)$. This is the infinite dimensional analogue of the (Everywhere) Missing at Random definition in \cite{seaman_what_2013}. The MAR assumption is a classical assumption in the literature on inference with missing data \cite{seaman_what_2013}, and has also been carried over to the non-parametric context \cite{yang_asymptotic_2020, cheng_nonparametric_1994}. For any $ n \in \mathbb{N}, m \leq n, w \in \Omega$ and $ M \in \mathcal{M}$ there are various methods available for estimating $\mathbb{E}_{n,m} [ \chi_{n,m}(w', \cdot) | O(\chi_{n,m}'(w', M, \cdot)) = O(\chi_{n,m}'(w, M, \cdot))]$ when the MAR assumption is satisfied (\textit{see, e.g.,} \cite{cheng_nonparametric_1994}).

By definition, one can see that if $\chi_{n,m}'$ is MAR, then so are the random elements over $\Omega \times \mathcal{M}$, $\chi_{n,m}'(w, M, \cdot)^{2}$ and $\chi_{n,m}'(w, M, \cdot, \cdot)$, and these elements may be analogously estimated. In the latter case, $\chi_{n,m}'(w, M, \cdot, \cdot)$ is considered observed only on the places where both terms of the tensor product are observed, and missing otherwise. 

Moreover, note that since $w \in \Omega$ and $M \in \mathcal{M}$ determine the set, 

\begin{align*}
    \{w' \in \Omega: O(\chi_{n,m}'(w', M, \cdot)) = O(\chi_{n,m}'(w, M, \cdot))\}    
\end{align*}

for all $n \in \mathbb{N}, m \leq n$, the following shortened notation can intuitively be used:

\begin{align*}
    \mathbb{E}_{n,m} [ \chi_{n,m}(w', \cdot) |  O(\chi_{n,m}'(w', M, \cdot)) = O(\chi_{n,m}'(w, M, \cdot))] 
    \\
    = \mathbb{E}_{n,m} [ \chi_{n,m}(w', \cdot) | w, M]
\end{align*}

and analogously for $\chi_{n,m}'(w, M, \cdot)^{2}$ and $\chi_{n,m}'(w, M, \cdot, \cdot)$ as random elements over $(w, M) \in \Omega \times \mathcal{M}$. 


Note that just as in the completely-observed case, there is a valid inner product for the $\{\chi_{n,m}'\}_{n \in \mathbb{N}, m \leq n}$ random elements, which exist on the separable Hilbert Space $\Omega \times \{0,1\}^{[0,1]}$. Thus, the analogous Lindeberg\textendash Feller and Lyapunov conditions in $H'$ will imply asymptotic normality of the $H'$ random variables. This is the second main result of this paper. The proofs are omitted, as the results follow immediately from this observation. 

\begin{definition}[The Lindeberg\textendash Feller Condition on Partial Hilbert Spaces]
$\forall \varepsilon > 0, \\ \lim_{n \to \infty} \sum_{m=1}^n \mathbb{E}'_{n,m}[ \|\chi_{n,m}' \|_{L^2(\mu)}^2 :  \|\chi_{n,m}' \|_{L^2( \mu)} > \varepsilon ] = 0$
\end{definition}

\begin{theorem}[The Lindeberg\textendash Feller Theorem on Partial Hilbert Spaces]
\label{thm:parlfnorm}
Suppose
\begin{enumerate}
    \item $\exists \varsigma' \in L^2 (\mu \times \mu) : \sum_{m=1}^n \mathcal{K}_{n,m}'$ converges to $\varsigma'$ with respect to \\ $\| \cdot \|_{L^2( \mu \times \mu )}$, and $\lim_{n \to \infty}\sum_{m=1}^n \mathbb{E}_{n,m}'[\|\chi_{n,m}'\|_{L^2 (\mu)}^2] > 0$. 
    
    \item $\forall \varepsilon > 0,  \lim_{n \to \infty} \sum_{m=1}^n \mathbb{E}_{n,m}' [ \|\chi_{n,m}' \|_{L^2(\mu)}^2 :   \|\chi_{n,m}' \|_{L^2( \mu)} > \varepsilon ] = 0 $ \\ (The Lindeberg\textendash Feller Condition on Partial Hilbert Spaces).
\end{enumerate}

Then, $S_n' = \sum_{m=1}^n\chi_{n,m}'$ is asymptotically normal with expectation $\mathbf{0}$ and covariance $\varsigma'$.
\end{theorem}

\begin{definition}[The Lyapunov Condition on Partial Hilbert Spaces]
There exists $\delta' > 0$ such that $\lim_{n \to \infty} \sum_{m=1}^{n} \mathbb{E}_{n,m}' [\|\chi_{n,m}'\|_{L^2(\mu)}^{2 + \delta'}] = 0$.
\end{definition}

\begin{theorem}[The Lyapunov Theorem on Partial Hilbert Spaces]
\label{thm:parlyap}
If there exists $\delta' > 0$ such that $\lim_{n \to \infty} \sum_{m=1}^{n} \mathbb{E}_{n,m}' [\|\chi_{n,m}'\|_{L^2( \mu)}^{2 + \delta'}] = 0$, then $\forall \varepsilon > 0$, 

\begin{align*}
    \lim_{n \to \infty} \sum_{m=1}^n \mathbb{E}_{n,m}' [ \|\chi_{n,m}' \|_{L^2( \mu)}^2 :  \|\chi_{n,m}' \|_{L^2(\mu)} > \varepsilon ] = 0
\end{align*}

\end{theorem}

An interesting follow-up question is therefore whether asymptotic normality of the partially observed random variables implies the asymptotic normality of the completely observed random elements, and if so, whether the inferences drawn from the partial functional data will mirror those that would have been drawn from complete data. This turns out to be the case. 

The following Theorems \ref{thm:lyapreg} and \ref{thm:cfe} together make up the third main result of this paper: that satisfaction of the Lindeberg\textendash Feller or Lyapunov conditions in the partial observation space implies their satisfaction in the complete space, and that the distributions of the partially observed and completely observed random variables will be asymptotically equivalent, meaning that \textit{all} possible inferences drawn from the former will be consistent. Theorem \ref{thm:lyapreg} shows that if the first condition of Theorem \ref{thm:parlfnorm} is satisfied in the partial observation space, then it is also satisfied in the complete observation space. In particular, this shows that the covariance of the partial-element distribution and the complete-element distribution will be identical. Theorem \ref{thm:cfe} shows that if the Lindeberg\textendash Feller condition is satisfied in the space of partially observed random variables, then it is satisfied in the complete space. Therefore, the distributions of the partially imputed data and complete data will both be asymptotically Gaussian with identical parameters. 

\begin{theorem}
\label{thm:lyapreg}
\begin{enumerate}

    \item If $\exists \varsigma \in L^2(\mu \times \mu) : \sum_{m=1}^n  \mathcal{K}_{n,m}'$ converges to $\varsigma$ with respect to $\| \cdot \|_{L^2( \mu \times \mu)}$, then $\sum_{m=1}^n \mathcal{K}_{n,m}$ converges to $\varsigma$ with respect to $\| \cdot \|_{L^2(\mu \times \mu)}$.
    
    \item If $\lim_{n \to \infty} \sum_{m=1}^n \mathbb{E}_{n,m}'[\|\chi_{n,m}' \| ^2_{L^2(\mu)}] < \infty$, then $\lim_{n \to \infty} \sum_{m = 1}^n \mathbb{E}_{n,m} [ \| \chi_{n,m}  \|^2_{L^2(\mu)}] < \infty$.
\end{enumerate}
\end{theorem}

\begin{theorem}
\label{thm:cfe}
If 

\begin{align*}
    \forall \varepsilon > 0,  \lim_{n \to \infty} \sum_{m=1}^n \mathbb{E}_{n,m}' [ \|\chi_{n,m}' \|_{L^2(\mu)}^2 :   \|\chi_{n,m}' \|_{L^2( \mu)} > \varepsilon ] = 0 
\end{align*}

then,

\begin{align*}
    \forall \varepsilon > 0,  \lim_{n \to \infty} \sum_{m=1}^n \mathbb{E}_{n,m}[ \|\chi_{n,m} \|_{L^2(\mu)}^2 :   \|\chi_{n,m} \|_{L^2(\mu)} > \varepsilon ] = 0 
\end{align*}
\end{theorem}

The key technique for proving these two results is to first condition the random elements with respect to $M \in \mathcal{M}$, then apply the smoothing law to an expectation over $\Omega$, and then aggregate the conditional expectations over $\mathcal{M}$. This is essentially a particular way of integrating over $\Omega \times \mathcal{M}$. This paper appears to be the first to apply this technique to the context of partially observed functional data, and shows how it may be useful in future research. This technique is illustrated in the Appendix. 

Theorems \ref{thm:lyapreg} and \ref{thm:cfe}, taken together, provide boundedness conditions which imply that the partially observed elements and the completely observed elements will have asymptotically identical distributions, because both will be asymptotically normal with the same covariance operator. Thus, any inferences drawn from the former, such as hypothesis tests results and confidence intervals \cite{kraus_components_2015, kraus_inferential_2019}, will approximate those that would have been drawn from the latter. 

\section{Appendix}
\textbf{Proof of Lemma \ref{lemma:taucont}:}\textit{If $\exists \varsigma \in L^2 (\mu \times \mu) : \sum_{m=1}^n \mathcal{K}_{n,m}$ converges to $\varsigma$ with respect to $\| \cdot \|_{L^2(\mu \times \mu)}$, then, $G: L^2(\mu \times \mu) \to \mathbb{R}$ such that $G(\psi) = \langle \varsigma, \psi \rangle _{L^2(\mu \times \mu)}$ is uniformly $\tau$-continuous near zero.}

\begin{proof}
By Section 1 (example 1) in \cite{gross_integration_1960}, it is sufficient to show that $G = r(K\psi)$ where $r: H \to \mathbb{R}$ such that (1) $r$ is uniformly continuous in the norm topology on bounded sets, and (2) $K$ is a Hilbert-Schmidt Operator. 

\begin{enumerate}
    \item Let $\psi \in H, r(\cdot) = \int_{[0,1]} (\cdot)(\psi)$. Let $\varepsilon >0.$ $\forall c_1, c_2 \in H : \|c_1 - c_2\|_{L^2(\mu)} < \frac{\varepsilon}{\|\psi\|_{L^2(\mu)}},$ 
    
    \begin{align*}
        |r(c_1) - r(c_2)| 
        \\ = \bigg| \int_{[0,1]}c_1\psi - \int_{[0,1]}c_2\psi \bigg| 
        \\\leq \int_{[0,1]} |(c_1 - c_2) (\psi)| 
        \\\leq \|c_1 - c_2\|_{L^2(\mu)} \| \psi\|_{L^2(\mu)} \ \text{by H\"older's inequality}
        \\
        < \varepsilon
    \end{align*}
    Hence, $r$ is uniformly continuous in the norm topology on bounded sets. 
    
    \item Let $x \in [0,1], K_x(\cdot) = \int_{[0,1]} \varsigma(x,y) (\cdot(y)) dy$. Note that
    \begin{align*}
        \int_{[0,1]} \int_{[0,1]}|\varsigma|^2  \leq \|\varsigma\|_{L^2(\mu \times \mu)}^2
        \
        \text{by H\"older's Inequality}
        \\
        < \infty 
    \end{align*}
    
      Thus, $K$ is a Hilbert\textendash Schmidt integral operator. 
\end{enumerate}
\end{proof}

Prior to proving Theorem \ref{thm:lfnorm}, a few preliminary lemmas are needed. 

\begin{lemma}
\label{lemma:prod}
If $\max_{1 \leq j \leq n} |c_{j,n}| \to 0, \sum_{j=1}^n c_{j,n} \to \lambda,$ and $\sup_{n}\sum_{j=1}^n|c_{j,n}| < \infty$, then $\prod_{j=1}^n (1 + c_{j,n}) \to e^\lambda$.
\end{lemma}

\begin{proof}
Proof in Chapter 3 of \cite{durrett_probability_2019}.  
\end{proof}

\begin{lemma}
$\bigg|e^{ix} - \sum_{m=1}^n \frac{(ix)^m}{m!}\bigg| \leq \min \bigg(\frac{|x|^{n+1}}{(n+1)!}, \frac{2|x|^n}{n!}\bigg)$.
\end{lemma}

\begin{proof}
Proof in Chapter 3 of \cite{durrett_probability_2019}.
\end{proof}

\begin{lemma}
\label{lemma:mineq}
Let $\chi$ be an $H-$valued random variable, $\psi \in H$, and $\phi_{\chi}(\psi):= \mathbb{E}[e^{i\langle \psi, \chi \rangle}]$ be the characteristic functional of $\chi$. If $\mathbb{E}[\|\chi\|_{L^2(\mu)}^2] < \infty$, then $\phi_{\chi}(\psi) = 1 + i\mathbb{E}[\langle \psi,\chi \rangle_{L^2(\mu)}] - \frac{\mathbb{E}[\langle \psi,\chi \rangle_{L^2(\mu)}^2]}{2} + o(\|\psi\|_{L^2(\mu)}^2)$, where $o(\|\psi\|_{L^2(\mu)}^2)$ indicates a quantity $j(\|\psi\|_{L^2(\mu)})$ such that $\frac{j(\|\psi\|_{L^2(\mu)})}{\|\psi\|_{L^2(\mu)}^2} \to 0 $ as $\|\psi\|_{L^2(\mu)} \to 0$
\end{lemma}

\begin{proof}
This follows immediately from the previous Lemma, and by properties of the inner product.  
\end{proof}

\begin{lemma}
\label{lemma:modulus}
Let $z_1,...,z_n, w_1,..., w_n \in \mathbb{C}$ with modulus $\leq \theta$. Then, $|\prod_{m=1}^nz_m - \prod_{m=1}^nw_m |$ $ \leq \theta ^{n-1}\sum_{m=1}^n|z_m - w_n|$.
\end{lemma}

\begin{proof}
Proof in Chapter 3 of \cite{durrett_probability_2019}.
\end{proof}

\textbf{Proof of Theorem \ref{thm:lfnorm}:} \textit{Suppose \begin{enumerate}
    \item $\exists \varsigma \in L^2 (\mu \times \mu) : \sum_{m=1}^n \mathcal{K}_{n,m}$ converges to $\varsigma$ with respect to $\| \cdot \|_{L^2(\mu \times \mu)}$, and $\lim_{n \to \infty}\sum_{m=1}^n \mathbb{E}_{n,m}[\|\chi_{n,m}\|_{L^2 (\mu)}^2] > 0$. 
    \item $\forall \varepsilon > 0,  \lim_{n \to \infty} \sum_{m=1}^n \mathbb{E}_{n,m}[ \|\chi_{n,m} \|_{L^2(\mu)}^2 :   \|\chi_{n,m} \|_{L^2(\mu)} > \varepsilon ] = 0 $ \\ (The Lindeberg\textendash Feller Condition).
\end{enumerate} Then, $S_n = \sum_{m=1}^n\chi_{n,m}$ is asymptotically normal with expectation $\mathbf{0}$ and covariance $\varsigma$.}

\begin{proof}
With Lemmas \ref{lemma:grossconthm} and \ref{lemma:taucont}, this proof is now able to closely track the proof of the classical Lindeberg\textendash Feller Theorem given in Chapter 3 of  \cite{durrett_probability_2019}. Let $\phi_{n,m}(g) = \mathbb{E}_{n,m}[e^{i\langle g, \chi_{n,m} \rangle_{L^2(\mu)}}]$ be the characteristic functional of $\chi_{n,m}$. By Lemma \ref{lemma:grossconthm}, it suffices to show that: 

$\prod_{m=1}^n \phi_{n,m}(g) \to e^{- \frac{\langle g, \varsigma \rangle_{L^2(\mu)} ^2}{2}}$ in probability, since the limit is the characteristic functional of a Gaussian Process with expectation $\textbf{0}$.

Let $z_{n,m} = \phi_{n,m}(g)$ and $w_{n,m} = 1 - \frac{\mathbb{E}_{n,m} [\langle g, \chi_{n,m} \rangle_{L^2(\mu)}^2]}{2}$. By Lemma \ref{lemma:mineq},

\begin{align}
\label{eqn:link}
    \|z_{n,m} -  w_{n,m}\|_{L^2(\mu)} 
    \\
    \leq \mathbb{E}_{n,m}[\min \{|\langle g, \chi_{n,m} \rangle_{L^2 (\mu)} |^ 3 , 2| \langle g, \chi_{n,m} \rangle_{L^2 (\mu)} | ^2 \}] 
    \nonumber
    \\
    \leq \mathbb{E}_{n,m}[|\langle g, \chi_{n,m} \rangle_{L^2 (\mu)} | ^3 : \|\chi_{n,m}\|_{L^2 (\mu)} \leq \varepsilon ] 
    \nonumber
    \\ 
    + \mathbb{E}_{n,m}[2|\langle g, \chi_{n,m} \rangle_{L^2 (\mu)}  |^2 : \|\chi_{n,m} \|_{L^2 (\mu)} > \varepsilon ] 
    \nonumber
    \\
    \leq \mathbb{E}_{n,m}[ \|g\|_{L^2 (\mu)}^3 \|\chi_{n,m}\|_{L^2 (\mu)} ^3 : \|\chi_{n,m}\|_{L^2 (\mu)} \leq \varepsilon ] 
    \nonumber
    \\ 
    + \mathbb{E}_{n,m}[2 \|g\|_{L^2 (\mu)}^2 \|\chi_{n,m}\|_{L^2 (\mu)} ^2 : \|\chi_{n,m} \|_{L^2 (\mu)} > \varepsilon ] 
    \nonumber
    \\
    \text{by the Cauchy-Schwarz Inequality}
    \nonumber
    \\
    \leq \|g\|_{L^2 (\mu)}^3\varepsilon \mathbb{E}_{n,m}[  \|\chi_{n,m}\|_{L^2 (\mu)} ^2 : \|\chi_{n,m}\|_{L^2 (\mu)} \leq \varepsilon ] 
    \nonumber
    \\ 
    + 2 \|g\|_{L^2 (\mu)}^2\mathbb{E}_{n,m}[ \|\chi_{n,m}\|_{L^2 (\mu)} ^2 : \|\chi_{n,m} \|_{L^2 (\mu)} > \varepsilon ]
    \nonumber
\end{align}

Summing from $m = 1$ to $n$, letting $n \to \infty$, and using the hypotheses of the Theorem, the following holds: 

\begin{align*}
    \limsup_{n \to \infty} \sum_{m=1}^n \| z_{n,m} - w_{n,m}\|_{L^2 (\mu)} \leq \limsup_{n \to \infty} \varepsilon \|g\|_{L^2 (\mu)}^3 \sum_{m=1}^n\mathbb{E}_{n,m}\big[\|\chi_{n,m}\|_{L^2 (\mu)} ^2\big]
\end{align*}

Since $\varepsilon$ is arbitrary, 

$$\lim_{n \to \infty} \sum_{m=1}^n \| z_{n,m} - w_{n,m}\|_{L^2 (\mu)} = 0 $$

The next step is to use Lemma \ref{lemma:modulus} with $\theta = 1$ to obtain: 

$$\bigg\| \prod_{m=1}^n \phi_{n,m}(g) - \prod_{m=1}^n w_{n,m}\bigg\|_{L^2 (\mu)} \to 0$$

To check the hypotheses of Lemma \ref{lemma:modulus}, note that since $\phi_{n,m}$ is a characteristic functional, $\|\phi_{n,m}\|_{L^2(\mu)} \leq 1$ for all $n,m$. For the terms in the second product, note that: 

\begin{align*}
    \mathbb{E}_{n,m} [\langle g, \chi_{n,m} \rangle_{L^2(\mu)}^2] =  \mathbb{E}_{n,m} [|\langle g, \chi_{n,m} \rangle_{L^2(\mu)}|^2]
    \\
    \leq \mathbb{E}_{n,m} [\|g\|_{L^2(\mu)}^2 \|\chi_{n,m}\|_{L^2(\mu)}^2] \ \text{by the Cauchy\textendash Schwarz inequality}
    \\
    \leq \|g\|_{L^2(\mu)}^2 (\varepsilon^2 + \mathbb{E}_{n,m} [ \|\chi_{n,m}\|_{L^2(\mu)}^2 : \|\chi_{n,m}\|_{L^2(\mu)} \geq \varepsilon])
\end{align*}

Further, (ii) implies that $\sup_{m \leq n} \mathbb{E}_{n,m}[\|\chi_{n,m}\|^2_{L^2(\mu)}] \to 0$ as $n \to \infty$, so $\varepsilon^2 + \mathbb{E}_{n,m} [ \|\chi_{n,m}\|_{L^2(\mu)}^2 : \|\chi_{n,m}\|_{L^2(\mu)} \geq \varepsilon] \to 0 $ as $n\to \infty$. Thus, as $n \to \infty$, $\| w_{n,m} \|_{L^2(\mu)} \leq 1 \ \forall m \leq n$.

Now, let $c_{n,m} = -\frac{1}{2}\mathbb{E}_{n,m}[\langle \chi_{n,m}, g \rangle_{L^2(\mu)}^2]$. As shown above, $$\sup_{m \leq n} \mathbb{E}_{n,m}[\|\chi_{n,m}\|^2_{L^2(\mu)}] \to 0$$ as $n \to \infty$. Moreover, by hypothesis 1, 

$$\lim_{n \to \infty}\sum_{m=1}^n c_{n,m} = -\frac{1}{2} \langle \varsigma, g \rangle _{L^2( \mu \times \mu)}$$

Thus, by Lemma \ref{lemma:prod},

$$\lim_{n \to \infty}\prod_{m=1}^n \bigg(1 -\frac{1}{2}\mathbb{E}_{n,m}[\langle \chi_{n,m}, g \rangle_{L^2(\mu)}^2]\bigg) = e^{\langle \varsigma, g \rangle _{L^2(\mu \times \mu)}}$$

in probability, where the right hand side is the characteristic functional of a Gaussian process. 

Applying Lemma \ref{lemma:grossconthm} completes the proof. 
\end{proof}

\textbf{Proof of Theorem \ref{thm:lyap}:} \textit{If there exists $\delta > 0$ such that $\lim_{n \to \infty} \sum_{m=1}^{n} \mathbb{E}_{n,m}[\|\chi_{n,m}\|_{L^2(\mu)}^{2 + \delta}] = 0$, then $\forall \varepsilon > 0,  \lim_{n \to \infty} \sum_{m=1}^n \mathbb{E}_{n,m}[ \|\chi_{n,m} \|_{L^2(\mu)}^2 :  \|\chi_{n,m} \|_{L^2(\mu)} > \varepsilon ] = 0$.}

\begin{proof}
Let $\varepsilon > 0$, $\chi \in H : \|\chi\|_{L^2(\mu)} > \varepsilon.$ Then, 
\begin{align*}
    \|\chi\|_{L^2(\mu)} ^2 = \frac{\|\chi\|_{L^2(\mu)}^{2 + \delta}}{\|\chi\|_{L^2(\mu)}^{ \delta}} \leq \frac{\|\chi\|_{L^2(\mu)}^{2 + \delta}}{\varepsilon^\delta}
\end{align*}

Thus, $\forall n \in \mathbb{N}, m \leq n$,

\begin{align*}
    \mathbb{E}_{n,m}[\|\chi_{n,m}\|_{L^2(\mu)}^2 : \|\chi_{n,m}\|_{L^2(\mu)} \geq \varepsilon] 
    \\
    \leq \mathbb{E}_{n,m}\bigg[\frac{\|\chi_{n,m}\|_{L^2(\mu)}^{2 + \delta}}{\varepsilon^\delta}\bigg]
\end{align*}

and therefore that

\begin{align*}
    \lim_{n \to \infty} \sum_{m=1}^{n} \mathbb{E}_{n,m}[\|\chi_{n,m}\|_{L^2(\mu)}^2 : \|\chi_{n,m}\|_{L^2(\mu)} \geq \varepsilon] 
    \\
    \leq \lim_{n \to \infty} \sum_{m=1}^{n} \mathbb{E}_{n,m}\bigg[\frac{\|\chi_{n,m}\|_{L^2(\mu)}^{2 + \delta}}{\varepsilon^\delta}\bigg]
    \\
    = \frac{1}{\varepsilon^\delta} \lim_{n \to \infty} \sum_{m=1}^{n} \mathbb{E}_{n,m}[\|\chi_{n,m}\|_{L^2(\mu)}^{2 + \delta}]
    \\
    = 0
\end{align*}
\end{proof}

\begin{lemma}
\label{lemma:eq1}
Let $n \in \mathbb{N}, m \leq n$. 

\begin{enumerate}
    \item $\mathcal{K}_{n,m}' = \mathcal{K}_{n,m}$
    
    \item$\mathbb{E}_{n,m}'[\| \chi_{n,m}' \|_{L^2(\mu)} ^ 2] = \mathbb{E}_{n,m}[\| \chi_{n,m} \|_{L^2(\mu)} ^ 2]$
\end{enumerate}
\end{lemma}

\begin{proof}
	Let $n \in \mathbb{N}, m \leq n$.
	
	\begin{enumerate}
		\item Let $(x_1, x_2) \in [0,1]^2$.
		
		\begin{align*}
		    \mathcal{K}_{n,m}'(x_1, x_2)
		    \\
			= \mathbb{E}_{n,m}' [\chi_{n,m}'(w,M,x_1) \chi_{n,m}'(w,M, x_2)]
			\\
			= \mathbb{E}_{n,m}'[(\chi_{n,m}(w,x_1) \mathbb{1}_{X_1(M)} (x_1) + (\mathbb{E}_{n,m}[\chi_{n,m} (w', x_1) | 		w,M] 
	        \\
	        + \mathcal{G}(0, \mathcal{K}_{n,m}^{w,M})(x_1))\mathbb{1}_{X_0(M)}(x_1)) (\chi_{n,m}(w, x_2) \mathbb{1}_{X_1(M)} (x_2) 	+ (\mathbb{E}_{n,m}[\chi_{n,m} (w', x_2) | w,M] 
	        \\
	        + \mathcal{G}(0, \mathcal{K}_{n,m}^{w,M}))(x_2)\mathbb{1}_{X_0(M)}(x_2))]
	        \\
	        = \mathbb{E}_{n,m}'[(\alpha + \beta ) (\gamma + \delta)]
		\end{align*}
		
		where  
	    \begin{align*}
	    \hfill
    		\alpha = \chi_{n,m}(w,x_1) \mathbb{1}_{X_1(M)} (x_1)
		    \\
        	\beta = (\mathbb{E}_{n,m}[\chi_{n,m} (w', x_1) | w,M] + \mathcal{G}(0, \mathcal{K}_{n,m}^{w,M})(x_1))\mathbb{1}_{X_0(M)}(x_1)
        	\\
        	\gamma = \chi_{n,m}(w, x_2) \mathbb{1}_{X_1(M)} (x_2) 
        	\\
        	\delta = (\mathbb{E}_{n,m}[\chi_{n,m} (w', x_2) | w,M] + \mathcal{G}(0, \mathcal{K}_{n,m}^{w,M})(x_2))\mathbb{1}_{X_0(M)}(x_2)
        \end{align*}
    
    Observe that:
    
    \begin{enumerate}
    	\item 
		\begin{align*}
			\mathbb{E}_{n,m}' [\alpha \gamma]
			\\
			=\mathbb{E}_{n,m}' [ \chi_{n,m}(w,x_1) \mathbb{1}_{X_1(M)} (x_1) \chi_{n,m}(w, x_2) \mathbb{1}_{X_1(M)} (x_2) ]
			\\
			= \mathbb{E}_{\mathcal{M}} [\mathbb{1}_{X_1(M)} (x_1) \mathbb{1}_{X_1(M)}(x_2) \mathbb{E}_{n,m} [\chi_{n,m}(w, x_1) \chi_{n,m}(w, x_2) | M ]]
		\end{align*}
		
	\item
		\begin{align*}
			\mathbb{E}_{n,m}'[\alpha \delta]
			\\
			=\mathbb{E}_{n,m}'[ \chi_{n,m}(w,x_1) \mathbb{1}_{X_1(M)} (x_1) (\mathbb{E}_{n,m}[\chi_{n,m} (w', x_2) | w,M] 
			\\
			+ \mathcal{G}(0, \mathcal{K}_{n,m}^{w,M})(x_2))\mathbb{1}_{X_0(M)}(x_2)]
			\\
			=\mathbb{E}_{\mathcal{M}} [ \mathbb{1}_{X_1(M)}(x_1) \mathbb{1}_{X_0(M)}(x_2) \mathbb{E}_{n,m}[(\chi_{n,m}(w,x_1) )(\mathbb{E}_{n,m}[\chi_{n,m} (w', x_2) | w,M] 
			\\
			+ \mathcal{G}(0, \mathcal{K}_{n,m}^{w,M})(x_2)) | M ]]
			\\
			=\mathbb{E}_{\mathcal{M}} [\mathbb{1}_{X_1(M)}(x_1) \mathbb{1}_{X_0(M)}(x_2) \mathbb{E}_{n,m}[\mathbb{E}_{n,m}[\chi_{n,m}(w,x_1) \chi_{n,m} (w', x_2) | w,M] 
			\\
			+ \chi_{n,m}(w,x_1)\mathcal{G}(0, \mathcal{K}_{n,m}^{w,M})(x_2) | M ]]
			\\
			=\mathbb{E}_{\mathcal{M}} [\mathbb{1}_{X_1(M)}(x_1) \mathbb{1}_{X_0(M)}(x_2) \mathbb{E}_{n,m}[\mathbb{E}_{n,m}[\chi_{n,m}(w,x_1) \chi_{n,m} (w', x_2) | w,M] | M ]]
			\\
			\text{by linearity and independence of $\chi_{n,m}(w,x_1)$ and $\mathcal{G}(0, \mathcal{K}_{n,m}^{w,M})(x_2)$}
			\\
			=\mathbb{E}_{\mathcal{M}} [\mathbb{1}_{X_1(M)}(x_1) \mathbb{1}_{X_0(M)}(x_2) \mathbb{E}_{n,m}[\chi_{n,m}(w,x_1) \chi_{n,m} (w, x_2) | M ]
			\\
			\text{by the Smoothing Law.}
		\end{align*}
		
		\item Analogously, 
			\begin{align*}
				\mathbb{E}_{n,m}'[\beta \gamma]
				\\
				= \mathbb{E}_{\mathcal{M}} [\mathbb{1}_{X_0(M)}(x_1) \mathbb{1}_{X_1(M)}(x_2) \mathbb{E}_{n,m}[\chi_{n,m}(w,x_1) \chi_{n,m} (w, x_2) | M ]
			\end{align*}
		
		\item 
			\begin{align*}
				\mathbb{E}_{n,m}'[\beta \delta]
				\\
				= \mathbb{E}_{n,m}'[ (\mathbb{E}_{n,m}[\chi_{n,m} (w', x_1) | w,M] + \mathcal{G}(0, \mathcal{K}_{n,m}^{w,M})(x_1))\mathbb{1}_{X_0(M)}(x_1) 
				\\
				 (\mathbb{E}_{n,m}[\chi_{n,m} (w', x_2) | w,M] + \mathcal{G}(0, \mathcal{K}_{n,m}^{w,M})(x_2))\mathbb{1}_{X_0(M)}(x_2)] 
				 \\
				 =\mathbb{E}_{\mathcal{M}} [\mathbb{E}_{n,m} [\mathbb{1}_{X_0(M)}(x_1)\mathbb{1}_{X_0(M)}(x_2) (\mathbb{E}_{n,m} [\chi_{n,m}(w', x_1) |w,M]  \mathbb{E}_{n,m} [
				 \\
				\chi_{n,m}(w',x_2)|w,M] + \mathcal{G}(0, \mathcal{K}_{n,m}^{w,M})(x_1)\mathcal{G}(0, \mathcal{K}_{n,m}^{w,M})(x_2)) |M]] 
				 \\
				 \text{by linearity and independence of $\mathbb{E}_{n,m} [\chi_{n,m}(w',x_2)|w,M] $} 
				 \\
				 \text{and $\mathcal{G}(0, \mathcal{K}_{n,m}^{w,M})(x_1)$, and likewise $\mathbb{E}_{n,m} [\chi_{n,m}(w', x_1) |w,M]$}
				 \\
				 \text{and $\mathcal{G}(0, \mathcal{K}_{n,m}^{w,M})(x_2)$}
				 \\
				 =\mathbb{E}_{\mathcal{M}} [ \mathbb{1}_{X_0(M)}(x_1)\mathbb{1}_{X_0(M)}(x_2) \mathbb{E}_{n,m} [\chi_{n,m} (w, x_1) \chi_{n,m} (w, x_2) | M ] ]
				 \\
				 \text{by the Smoothing Law and by properties of the covariance}
				 \\
				 \text{operator (Theorem 7.2.4 in \cite{hsing_theoretical_2015}).}
			\end{align*}
	Thus, 
	
	\begin{align*}
		\mathbb{E}_{n,m}'[(\alpha + \beta ) (\gamma + \delta)]
		\\
		= \mathbb{E}_{n,m} [\chi_{n,m} (w, x_1) \chi_{n,m} (w, x_2)] 
		\\
		\text{by linearity and the Smoothing Law.}
		\\
		= \mathcal{K}_{n,m}(x_1, x_2)
	\end{align*}
    \end{enumerate}
	
	\item 
		\begin{align*}
 		\mathbb{E}_{n,m}'[\| \chi_{n,m}' \|_{L^2(\mu)} ^ 2] 
        	\\
        	= \int_{[0,1]} \mathbb{E}_{n,m}' [ (\chi'_{n,m})^{2} ] \ \text{by the Fubinin-Tonelli Theorem}
        	\\
	        = \int_{[0,1]} \mathbb{E}_{n,m}' [( \chi_{n,m} (w, \cdot) \mathbb{1}_{X_1(M)}(\cdot) 
	        \\
	        + (\mathbb{E}_{n,m}[\chi_{n,m} (w', \cdot) | w,M] + \mathcal{G}(0, \mathcal{K}_{n,m}^{w,M})(\cdot)) \mathbb{1}_{X_0(M)}(\cdot))^2]
	        \\
	        = \int_{[0,1]} ( \mathbb{E}_{n,m}' [ \chi_{n,m} (w, \cdot) ^2 \mathbb{1}_{X_1(M)}(\cdot)]
	        \\
	        + \mathbb{E}_{n,m}' [ (\mathbb{E}_{n,m}[\chi_{n,m} (w', \cdot) | w,M] + \mathcal{G}(0, \mathcal{K}_{n,m}^{w,M})(\cdot))^2 \mathbb{1}_{X_0(M)} (\cdot)] )
	        \\
	        \text{by linearity and the fact that $X_1(M)$ and $X_0(M)$ are disjoint for all $M \in \mathcal{M}$.}
	        \\
	        = \int_{[0,1]} ( \mathbb{E}_{n,m}' [ \chi_{n,m} (w, \cdot) ^2 \mathbb{1}_{X_1(M)}(\cdot)]
	        \\
	        + \mathbb{E}_{n,m}' [( \mathbb{E}_{n,m}[\chi_{n,m} (w', \cdot) | w,M]^2 + \mathcal{G}(0, 			\mathcal{K}_{n,m}^{w,M})(\cdot)^2) \mathbb{1}_{X_0(M)}(\cdot)] )
	        \\
	        \text{by linearity and independence}
	        \\
	        = \int_{[0,1]} ( \mathbb{E}_{n,m}' [ \chi_{n,m} (w, \cdot) ^2 \mathbb{1}_{X_1(M)}] + \mathbb{E}_{n,m}' [ \mathbb{E}		_{n,m}[\chi_{n,m} (w', \cdot) ^2| w,M]\mathbb{1}_{X_0(M)}] )
	        \\
	        \text{by properties of the covariance operator (Theorem 7.2.4 in \cite{hsing_theoretical_2015}).}
	        \\
	        = \int_{[0,1]} ( \mathbb{E}_{\mathcal{M}}[\mathbb{E}_{n,m} [ \chi_{n,m} (w, \cdot)^2 \mathbb{1}_{X_1(M)} | M] 
	        \\
	        + \mathbb{E}_{\mathcal{M}} [ \mathbb{E}_{n,m} [  \mathbb{E}_{n,m} [ \chi_{n,m} (w', \cdot)^2 \mathbb{1}_{X_0(M)}| w,M]  |M ] ]])
	        \\
	        = \int_{[0,1]} \mathbb{E}_{\mathcal{M}}[\mathbb{E}_{n,m} [ \chi_{n,m} (w, \cdot)^2 \mathbb{1}_{X_1(M)} | M] + 	\mathbb{E}_{n,m} [ \chi_{n,m} (w, \cdot)^2 \mathbb{1}_{X_0(M)}|M ] ]
	        \\
	        \text{by linearity and the Smoothing Law}.
 	       \\
 	       = \int_{[0,1]} \mathbb{E}_{\mathcal{M}}[\mathbb{E}_{n,m} [ \chi_{n,m} (w, \cdot)^2 \mathbb{1}_{X_1(M)} + \chi_{n,m} (w, \cdot)^2 \mathbb{1}_{X_0(M)}|M ] ]
 	       \\
 	       \text{by linearity}.
 	       \\
	       	 =\mathbb{E}_{n,m}[\| \chi_{n,m} \|_{L^2(\mu)} ^ 2]
	        \\
	        \text{by the Smoothing Law and Fubini\textendash Tonelli Theorem.}
		\end{align*}
	
	\end{enumerate}
\end{proof}

\textbf{Proof of Theorem \ref{thm:lyapreg}:} 
\begin{enumerate}
    \item If $\exists \varsigma \in L^2(\mu \times \mu) : \sum_{m=1}^n  \mathcal{K}_{n,m}'$ converges to $\varsigma$ with respect to $\| \cdot \|_{L^2( \mu \times \mu)}$, then $\sum_{m=1}^n \mathcal{K}_{n,m}$ converges to $\varsigma$ with respect to $\| \cdot \|_{L^2(\mu \times \mu)}$.
    
    \item If $\lim_{n \to \infty} \sum_{m=1}^n \mathbb{E}_{n,m}'[\|\chi_{n,m}' \| ^2_{L^2(\mu)}] < \infty$, then $\lim_{n \to \infty} \sum_{m = 1}^n \mathbb{E}_{n,m} [ \| \chi_{n,m}  \|^2_{L^2(\mu)}] < \infty$.
\end{enumerate}

\begin{proof} 
These results follow immediately from Lemma \ref{lemma:eq1}

\end{proof}

\textbf{Proof of Theorem \ref{thm:cfe}:} 

If 

\begin{align*}
    \forall \varepsilon > 0,  \lim_{n \to \infty} \sum_{m=1}^n \mathbb{E}_{n,m}' [ \|\chi_{n,m}' \|_{L^2(\mu)}^2 :   \|\chi_{n,m}' \|_{L^2( \mu)} > \varepsilon ] = 0 
\end{align*}

then,

\begin{align*}
    \forall \varepsilon > 0,  \lim_{n \to \infty} \sum_{m=1}^n \mathbb{E}_{n,m}[ \|\chi_{n,m} \|_{L^2(\mu)}^2 :   \|\chi_{n,m} \|_{L^2(\mu)} > \varepsilon ] = 0 
\end{align*}

\begin{proof}
By Lemma \ref{lemma:mineq}, it is sufficient to show that for any $\psi \in H$, the following hold 

\begin{enumerate} 

	\item $\mathbb{E}_{n,m}'[\| \chi_{n,m}' \|_{L^2(\mu)} ^ 2] < \infty$, $\implies$ $\mathbb{E}_{n,m}[\| \chi_{n,m} \|_{L^2(\mu)} ^ 2] < \infty$
	
	
	\item $\mathbb{E}_{n,m}' [ \langle \psi, \chi_{n,m}' \rangle_{L^2(\mu)} ^2 ]  = \mathbb{E}_{n,m} [ \langle \psi, \chi_{n,m} \rangle_{L^2(\mu)} ^2 ]$
 
\end{enumerate}

since this allows us to replace $\mathbb{E}_{n,m} [ \langle g, \chi_{n,m} \rangle_{L^2(\mu)} ^2 ]$ in the proof of Theorem \ref{thm:lfnorm} (beginning with equation \ref{eqn:link}) with $\mathbb{E}_{n,m}' [ \langle g, \chi_{n,m}' \rangle_{L^2(\mu)} ^2 ] $.

\begin{enumerate}
    \item 
    \begin{align*}
         \infty > \mathbb{E}_{n,m}'[\| \chi_{n,m}' \|_{L^2(\mu)} ^ 2] 
         \\
         = \mathbb{E}_{n,m}[\| \chi_{n,m} \|_{L^2(\mu)} ^ 2]
         \\
         \text{by Lemma \ref{lemma:eq1}.}
    \end{align*}
        
    \item
    \begin{align*}
        \mathbb{E}_{n,m}'[\langle \psi, \chi_{n,m}' \rangle_{L^2(\mu)}^2]
        \\
        = \mathbb{E}_{n,m}' \bigg[ \bigg(\int_{[0,1]} \psi \chi_{n,m}' \bigg)^2 \bigg]
        \\
        \mathbb{E}_{n,m}' \bigg[\int_{(x_1, x_2) \in [0,1]^2} \psi(x_2)\psi(x_1) \chi_{n,m}'(w,M, x_1)\chi_{n,m}'(w, M, x_2) \bigg]
        \\
        \text{by the Fubini\textendash Tonelli Theorem}
        \\
        = \int_{[0,1]^2} \psi(x_1)\psi(x_2) \mathbb{E}_{n,m}' [\chi_{n,m}'(w,M,x_1) \chi_{n,m}'(w,M, x_2)] 
        \\
        \text{by the Fubini\textendash Tonelli Theorem}
        \\
        = \int_{[0,1]^2} \psi(x_1)\psi(x_2) \mathbb{E}_{n,m} [\chi_{n,m}(w,M,x_1) \chi_{n,m}(w,M, x_2)] 
        \\
        \text{by Lemma \ref{lemma:eq1}}
        \\
        =\mathbb{E}_{n,m}[\langle \psi, \chi_{n,m} \rangle_{L^2(\mu)} ^2 ] 
        \\
    	\text{by the Fubini\textendash Tonelli Theorem.}
    \end{align*}
\end{enumerate}
\end{proof}

\bibliography{references}

\end{document}